\documentclass[12pt,reqno]{amsart} 

\usepackage{amsfonts,epsfig}

 \usepackage{hyperref,color}

\def\accentsfrancais{applemac}
   \RequirePackage[\accentsfrancais]{inputenc}

\def\R{\mathbb{R}}
\def\RR{\mathbb{R}}
\def\Q{\mathbb{Q}}
\def\QQ{\mathbb{Q}}
\def\N{\mathbb{N}}
\def\Z{\mathbb{Z}}

\newcommand{\BE}{\begin{equation}}
\newcommand{\EE}{\end{equation}}

\newcommand{\al}{\alpha}

\newcommand{\x}{x_0}
\newcommand{\pqn}{\frac{p_n}{q_n}}

\newcommand{\TB}{\mathfrak{B}}

\newcommand{\be}{\beta}
\newcommand{\hun}{h^1_B}
\newcommand{\cgot}{{\mathfrak c}}

\newsavebox{\fmbox}
\newenvironment{fmpage}[1]
 {\begin{lrbox}{\fmbox}\begin{minipage}{#1}}
 {\end{minipage}\end{lrbox}\fbox{\usebox{\fmbox}}}

\def\eps{\varepsilon}

\usepackage{geometry}
\theoremstyle{plain}
\newtheorem{thm}{Theorem}
\newtheorem{prp}{Proposition}
\newtheorem{lem}{Lemma}
\newtheorem{df}{Definition}
\newtheorem{cor}{Corollary}

\theoremstyle{remark}

\newcommand{\BL}{\begin{lem}}
\newcommand{\EL}{\end{lem}}

\newcommand{\BC}{\begin{cor}}
\newcommand{\EC}{\end{cor}}

\newcommand{\BP}{\begin{prp}}
\newcommand{\EP}{\end{prp}}

\newcommand{\BD}{\begin{df}}
\newcommand{\ED}{\end{df}}

\DeclareMathOperator*{\supess}{sup\,ess}

\title{Multifractal analysis of the  Brjuno function}
\author{St\'ephane Jaffard   and Bruno Martin} \thanks{ 
       Addresses: { \bf St\'ephane Jaffard}: Universit\'e Paris Est, Laboratoire d'Analyse et de Math\'ematiques Appliqu\'ees, CNRS UMR 8050, UPEC,  Cr\'eteil, France, 
       {\tt jaffard@u-pec.fr}
\\ 
      { \bf  Bruno Martin}: Universit\'e du Littoral C\^ote d'Opale, EA 2797, Laboratoire de Math\'ematiques Pures et Appliqu\'ees Joseph Liouville,  F-62228 Calais, France,
{\tt Bruno.Martin@univ-littoral.fr}
} 
\date{}
\address{}
\email{}

\begin{document}

\begin{abstract}  The Brjuno function $B$  is a 1-periodic, nowhere locally boun\-ded function,  introduced by J.-C. Yoccoz because it encapsulates a key information concerning  analytic small divisor problems in dimension 1.  We show that   $T^p_\al$ regularity,  introduced by  Calder\'on and Zygmund, is the only one which is relevant in order to unfold the pointwise regularity properties of $B$; we  
 determine its $T^p_\al$ regularity at every point and show that it is  directly related  to the irrationality exponent $\tau (x)$: its $p$-exponent at $x$ is exactly $1/\tau (x)$.   This new example of multifractal function puts into light a new link  between  dynamical systems and fractal geometry.    Finally we also determine the H\"older exponent of a primitive of $B$. 
 \end{abstract}

\maketitle

{ \small {\em Keywords:} Brjuno function, continued fractions, irrationality exponent, Diophantine approximation, $T^p_\al$ regularity,  $p$-exponent, multifractal analysis.\\ 

} \\

 { \bf Acknowledgements:  }The authors   thank Yves Meyer  and   the  anonymous referee for many remarks on  previous versions of this text. 
\\

St\'ephane Jaffard is  supported by  ANR  Grant MULTIFRACS ANR-16-CE33-0020 and Bruno Martin by ANR Grant MUDERA ANR-14-CE34-0009.

\section{Introduction}\label{sec:intro}
Let $x$ be an irrational number in $(0,1)$, and   let
  \BE \label{contfracexp} x = [ 0; a_1, \ldots, a_n , \ldots ]\EE
denote its  continued fraction expansion.   
The convergents $p_n / q_n$ of $x$ are 
 \[ \frac{p_n}{q_n} = \cfrac{1}{a_1+\cfrac{1}{a_2+\cfrac{1}{\ddots+\cfrac{1}{a_n}}}} \]
(following the standard tradition,  we will not write explicitly the dependency of $p_n$, $q_n$ and $a_n$ in $x$ except when it will be  needed). 
The Brjuno function at $x$  is 
 \BE \label{B2}  B(x) =  \sum_{n=0}^\infty | p_{n-1} - q_{n-1} x|  \log \left( \frac{p_{n-1} - x \; q_{n-1}}{q_n x -p_n}\right) ,  \EE
 where,   by convention,  
 \[ (p_{-1}, q_{-1}) = (1,0), \quad (p_{0}, q_{0}) = (0,1), \quad \mbox{ and} \quad (p_{1}, q_{1}) = (1,a_1),\]
 so that the first term in \eqref{B2} is $\log (1/x)$. 
The Brjuno function is extended by periodicity on  $\R - \Q$.

The Brjuno function plays an important role  in the theory of holomorphic dynamical systems: It was first introduced by J.-C. Yoccoz, see \cite{Yo}, because of the   information that it yields concerning  analytic small divisor problems in dimension 1:  Following C. L. Siegel \cite{Sieg}, A. D.  Brjuno \cite{Brju} and J.-C. Yoccoz \cite{Yo},  germs with linear part $e^{2i\pi  x}$  are analytically conjugate to a rotation if and only if $x$ is a {\em Brjuno number}, i.e. if   $x\notin \Q $  and   if  the series defining $B(x)$ is convergent.  
 
 B. Marmi, P. Moussa and J.-C. Yoccoz  determined the optimal global regularity of $B$, showing that  it belongs to  $  BMO$, see \cite{MaMoYo}. 
 The  {\em Marmi-Moussa-Yoccoz conjecture}  is another regularity problem related with the Brjuno function: It states that the sum of $B$ and  the logarithm of the conformal radius of the Siegel disk of a monic quadratic polynomial is $C^{1/2}$, see \cite{MaMoYo} p. 267.
Key steps towards its resolution 
   have been obtained by X. Buff, D. Cheraghi and A.  Ch\'eritat, see \cite{BuChe,CheChe}.
 Local properties of $B$  were recently investigated by M. Balazard and B. Martin in \cite{BaMa}:  They showed that its Lebesgue points  are precisely the Bruno numbers,  and they obtained precise estimates of the average of $B$ over an interval, which will play a key-role in our study, see e.g \eqref{BMfond}.

 We will complement these regularity results by performing the {\em multifractal analysis}  of the Brjuno function. The   multifractal analysis of a function $f$ usually consists into three steps:
 \begin{itemize}
\item Choose  a {\em pointwise regularity  exponent} compatible with the  global function space setting where $f$ is considered, 
\item  determine the value taken by this  exponent at every point,  
\item compute the Hausdorff dimensions ${ \mathcal D}_f(H)$ of the sets of points where this exponent takes a given value $H$. 
\end{itemize}
The function $H \rightarrow { \mathcal D}_f(H)$ 
is  the {\em multifractal spectrum} of $f$.
Multifractal analysis  has  also been developped in the setting of  measures and even of distributions, see e.g. \cite{BBBFHJMP,Mey,Pesin} and references therein.

Several clues indicate that the tools supplied by multifractal analysis are relevant for the  Brjuno function: First it is a {\em cocycle} under the action of  $PGL (2, \Z)$, as a consequence of  the remarkable functional equations
\begin{align*}
\forall x \in \R\setminus\Q,& \quad B(x+1) =B(x), \\ 
\forall x \in(0,1)\setminus\Q,& \quad B(x) = \log(1/x) + xB(1/x),
\end{align*}
see  \cite{MaMoYo2,MaMoYo3}. 
This property is reminiscent of the  behavior of the Jacobi theta function  under modular transforms, which  is  the key ingredient in the determination of the pointwise exponent of the {\em non-differentiable Riemann function}  ${ \mathcal R} (x) = \sum \sin (\pi n^2x) / n^2$  \cite{JaRie}, and  of  related trigonometric series \cite{SeuUb}. Other trigonometric series also related to modular forms have  been studied  by  I. Petrykiewicz in  \cite{Petry1,Petry2}. Finally, \eqref{B2}  also indicates  that  Diophantine approximation properties  should  play a role in the  local regularity properties of $B$.  
 This was  the case for  ${ \mathcal R}$ \cite{JaRie}, and  several of its generalizations   investigated by   F. Chamizo, I. Petrykiewicz, S. Ruiz-Cabello, and A.  Ubis  in  \cite{Cha,ChaPeRu,ChaUb}, and by T.~Rivoal and J.~Roques in \cite{RivRoq}. Note that extremely few explicit deterministic functions playing an important role in mathematics have been proved to have a non-trivial  multifractal spectrum: Most results in multifractal analysis are either  of probabilistic or generic nature. Another motivation for performing such an analysis on $B$ is that, beyond the important role played by  this function, our result establishes a new relationship between holomorphic dynamical systems on one side, and real analysis and geometric measure theory on the other.

In order to perform the multifractal analysis of $B$, a first  question  is to determine a pointwise exponent  fitted to its study.
As mentioned above, this will be a consequence of the  choice of  a right function space setting. 
The two notions of pointwise regularity  most commonly used are the  {\em H\"older exponent},  defined for locally bounded functions  and the {\em local dimension},  defined for positive  Radon measures (see Section \ref{sec:addit}). However, these exponents are not fitted to the analysis of the Brjuno function for the following reasons. First, $B$ is not locally bounded (i.e. does not coincide a.e. with a locally bounded function), because of the logarithmic singularities in \eqref{B2} centered at all rational points (the series \eqref{B2} is positive so that cancellations between terms cannot occur). As regards the local dimension, since $B$ is positive, we can interpret it as the density of a positive Radon measure, but its local dimension is constant  so that it is not adapted to measure  the variations of regularity that exist  in $B$. On other hand, these variations will be put into light through the use of a third notion of pointwise regularity, introduced by Calder\'on and Zygmund see \cite{CZ}, which is fitted to the study of functions that belong to $L^p_{loc}$.  
 
\begin{df}\label{df:Calder} Let $p \in [1, +\infty)$ and $\al \geq -1/p$.  Let $f\in L^p_{loc}(\R)$, and $x_0\in\R$;  $f$  belongs to $T^p_\al (x_0)$ if there exist $C >0$  and a polynomial $P$ of degree less than $\al$ (with $P\equiv 0$ if $\al<0$)  such that, for $\rho$ small enough,   
\BE\label{deux}  
\left(\frac{1}{2\rho} 
\int_{x_0-\rho}^{x_0+ \rho}|f(x)-P(x-x_0)|^p dx 
\right)^{1/p}
\leq C\rho^\al . 
\EE
The
$p$-exponent of  $f$ at $x_0$ is 
\[
h^{p}_f
(x_0)= \sup \{ \al : f \in T^p_\al (x_0)\}.
\]
\end{df}

{ \bf Remarks:} 

\begin{itemize}
\item The normalization chosen in (\ref{deux}) is such that  the simple { \em cusp singularities }  $|x-x_0|^\al$ have an H\"older and a $p$-exponent which take the same value $\alpha$ at $x_0$ (for any $p \geq 1$). 
\item Definition \ref{df:Calder}    is a natural substitute for pointwise H\"older regularity when
 functions in 
$L^p_{loc}$ are considered. In particular, the $p$-exponent can take negative values  down to $-d/p$, and typically  allows to deal with  singularities  which are locally of the form 
${1}/{|x-x_0|^\gamma }$  for $\gamma < d/p$.
\item The condition on the degree of $P$ (which is required to ensure uniqueness of $P$) implies that, if $\al \leq 0$, then  $P =0$.  
\end{itemize}

Let $p=1$; if $f\in L^1_{loc}$, and if the left-hand side of \eqref{deux} is a $o(1)$, then 
$x_0$ clearly is a Lebesgue point of $f$ and the constant term of $P$ is the Lebesgue value of $f$ at $x_0$, i.e. is 
\BE\label{leb} 
\lim_{\rho\rightarrow 0}\;\frac{1}{2\rho}\int_{x_0-\rho}^{x_0+ \rho}f(x)\;dx. 
\EE
Indeed, if we denote by $D$ this constant term, 
\[
 \left| \frac{1}{2\rho}\int_{x_0-\rho}^{x_0+ \rho}  f(x)   dx -D \right| \leq \frac{1}{2\rho}\int_{x_0-\rho}^{x_0+ \rho}  \left|   f(x)-D \right|   dx = o (1). \] 

It follows that the 1-exponent measures the rate of convergence of the local averages \eqref{leb} in the Lebesgue theorem. Therefore the determination of the 1-exponent that we will perform can be interpreted as a  quantitative sharpening of the theorem of M. Balazard and B. Martin stating that every Brjuno point is a Lebesgue point of the Brjuno function. 
This  is our main motivation for  focusing  on the case $p=1$.  However,  in Section \ref{secpexps} we will deal with arbitrary $p$s (and conclude that, at any point, the $p$-exponent is independent of $p$). The 1-exponent of $B$ at a point will be related with its {\em  (Diophantine) irrationality exponent}.  

\BD 
Let $\x\notin \Q$, 
and $p_n / q_n$  the sequence of convergents of the continued fraction expansion of $\x$. Let $\tau_n (x_0) $  be defined by
\[
\left|\x-\pqn\right| =\frac{1}{q_n^{\tau_n (x_0) }}. 
\]
The irrationality exponent (also called Diophantine approximation exponent or Diophantine order) of $x_0$ is 
\[ 
\tau (x_0) = \limsup_{n \rightarrow + \infty}  \tau_n (x_0) . 
\]
\ED

If $x_0$ is irrational, then  $|x_0 -{p_n}/{q_n}| < {1}/{q_n^2}$,  so that $\tau_n (x_0)>2$, and $\tau (x_0) \geq 2$.
Let us recall the following equivalent definition for the irrationality exponent of $x_0$ : $\tau(x_0)$ is the supremum of the $\tau\in\R$ such that there exists infinitely many $(p,q)\in\Z\times \N^*$ such that 
$ |x-p/q| \le 1/q^{\tau}$. 

\begin{thm}\label{thm:SpecB}
If $x_0 \in \QQ$, then $h^1_B (x_0) = 0$. Otherwise, 
\[ h^1_B (x_0) =  \frac{1}{\tau (x_0)}. 
\] 
\end{thm}

{\bf Remark:} Since almost every real number $x$ satisfies $\tau (x) =2$ (see e.g. Chap. 10.3 of  \cite{Falconer:2003oj})
it follows that $h^1_B$ takes the value $1/2$ almost everywhere. In the opposite direction, since quasi-every real number (in the sense of Baire categories) satisfies $\tau (x) =+ \infty$,   see \cite{Oxtob},  it follows that $h^1_B$ vanishes  quasi-everywhere (i.e. vanishes at least on a countable intersection of open dense subsets).

We now derive the consequence of Theorem \ref{thm:SpecB} for multifractal analysis. Let $\dim (A) $ denote the Hausdorff dimension of the set $A$, with the convention $\dim(\emptyset )=-\infty$.
  
\begin{df}\label{df:spectrum} 
Let   $p \in [1, +\infty)$  and  $f\in L^p_{loc}(\R)$. The level sets of $h^{p}_f$, denoted by $E^p_H $, are 
\[ \forall H \in \left[ -\frac{1}{p}, + \infty \right], \qquad  E^p_H = \{ x: \;\; h^{p}_f (x) =H\}. 
\]
The  $p$-spectrum  of $f$  is the function 
${\mathcal D}_f^p: \; \left[ -{1}/{p}, + \infty \right]  \to   \RR \cup \{-\infty\} $
defined by 
\[
{ \mathcal D}^p_f (H) =  \dim \; ( E_H^p  ). 
\]
\end{df}

In contradistinction with the H\"older case, few $p$-spectrums have been determined:  Let us mention the characteristic functions of some fractal sets  \cite{JaMel} and random wavelet series \cite{Bandt};   generic results (in the Baire and prevalence settings) for functions in a Sobolev space were obtained by A. Fraysse  \cite{Fraysse};  recently, 2-exponents of trigonometric series which are not locally bounded were obtained by S.~Seuret and A.~Ubis  \cite{SeuUb}.  

The precised formulation of Jarnik's theorem states that 
\BE \label{jar}
 \dim \left\{ x : \, \tau (x) = t \right\} = \frac{2}{t}, \EE
see e.g. \cite{JaRie}.  Therefore the 1-spectrum of $B$ will therefore follow from Theorem \ref{thm:SpecB}.  

\begin{cor}\label{corspec}
The  $1$-spectrum  of $B$  is 
\BE\label{corospec} 
{\mathcal D}^1_B(H)=
\begin{cases}
2H & \mbox{if } H\in [0,1/2],\\
-\infty & \mbox{else. }
\end{cases}
\EE
\end{cor}

{ \bf Remark:} Since \eqref{jar} also holds after restricting to  the points $x$ inside a nonempty open interval, it follows that the multifractal spectrum of $B$ restricted to any interval $(a,b)$  of positive length is also given by \eqref{corospec}. Following \cite{BaDuJaSeu}, $B$  is an {\em homogeneous multifractal function}. \\

Theorem \ref{thm:SpecB} is proved in Section \ref{sec2}.  The computation of the 1-exponent is sharpened in  Section \ref{secbad} where the exact modulus of continuity of $B$ at {\em badly approximable numbers} is determined. Results concerning other notions of pointwise regularity are grouped in Section \ref{sec:addit}. Finally, we mention related open problems in Section \ref{sec:conclu}.

\section{Determination of the $1$-exponent  of $B$ }
\label{sec2}

The  fact that $B \in BMO$ implies a uniform lower bound on the 1-exponent. Indeed it follows  from
 the John-Nirenberg inequality (or from Proposition 3 of \cite{BaMa}) that  
\BE \label{JohnNin} \exists C >0, \; \forall x,y :  \;\;  |x-y|\leq \frac{1}{2}, \qquad { \Big| \int_x^y  B(t)   \; dt \Big|}  \leq C |x-y| \log \left( \frac{1}{|x-y| }\right)  
\EE
(here and in the following, the value of the constant $C$  may change from one line to the next).
Thus, for $h < 1/2$, 
\[ \forall D, \qquad  \frac{1}{2h} \int_{x_0-h}^{\x +h} | B(x) -D | dx \leq C  \log (1/h ) 
\]
and finally, 
\BE \label{regunif} 
\forall \x , \qquad \hun  (\x ) \geq 0. 
\EE

Following \cite{BaMa}, it will be convenient to define a  function $\tilde{B}$  at rationals in the following way:  If $x_0 \in (0,1)\cap\Q $,  then the continued fraction expansion \eqref{contfracexp} of $x_0$ stops at a rank $N$, 
and 
\[ \tilde{B}(x_0) =   \sum_{n=0}^{N-1}\;  | p_{n-1} - q_{n-1} x_0| \;  \log \left( \frac{p_{n-1} - x_0 \; q_{n-1}}{q_n x_0 -p_n}\right)  ; 
\]
for instance,   for $N =1$,  $\tilde{B}(1/k) = \log k$. 
 
The  regularity of $B$ at rationals is a consequence of the following estimate of Balazard and Martin (Proposition  12 of \cite{BaMa}):
Let $r= p/q$ with $p\wedge q =1$; 
\BE \label{BMfond} 
\mbox{ if} \;  | h | < \frac{2}{3 q^2}, \;
\frac{1}{h } \int_r^{r+h} \hspace{-2mm} B(x) dx = \frac{\log (e / q^2 | h| )}{q} + \tilde{B}(r)  + { \mathcal O}    \left(  qh  \log \left(\frac{1}{q^2 |h|}\right)  \right) 
\EE
where the ${ \mathcal O} $ is uniform  (in $p$, $q$ and $h$).  In particular,  if $\x = p/q$ is  rational, then $\forall D$,    for $h$ small enough, 
\BE \label{irregrat}   \int_{x_0-h}^{\x +h} | B(x) -D | dx \geq  \frac{h}{2q} \log (1/h ) ,\EE
so that, at rationals $\hun (x_0) =0$. 
More precisely, by \eqref{JohnNin},  for $C$ large enough, the function $ C h\log (1/h)$ is a uniform   $1$-modulus of continuity of $B$ (which will be defined further at Definition \ref{defmodgen});
and, up to the multiplicative constant,  this   is optimal, because  it follows from \eqref{JohnNin} and \eqref{irregrat} that {$ h \log (1/h)$ is the order of magnitude of  the left hand side of \eqref{irregrat}}. 

 The regularity of $B$ at  Cremer numbers (i.e. at irrationals that are not Brjuno numbers)  follows from the fact that they are not  Lebesgue points, see   Proposition 14 of \cite{BaMa}.  
Thus, it follows from \eqref{leb} that 
$h^1_B (x_0) \leq 0$  and, using \eqref{regunif},   $h^1_B (x_0) = 0$. Therefore, from now on,  we can assume that $x_0$ is   a Brjuno number, so that $B(x_0)$ is  finite, and its   values (pointwise and  in the Lebesgue sense)  coincide. 
 
\subsection{Global and pointwise irregularity of the Brjuno function }

The idea for proving the irregularity of $B$ at Brjuno numbers  is to reinterpret \eqref{BMfond} as  implying that some of its wavelet coefficients  are large in the neighbourhood of the point considered, so that $B$ is irregular at those points.  We will need a  variant of the classical  wavelet criterium  (such as in  
 \cite{Jafcras}).

We assume in the following that $\psi$  is a bounded, compactly supported function  satisfying 
\BE \label{hyppsi}  \sup_{x\in \RR} | \psi (x) | \leq 1 \qquad \mbox{ and} \qquad  \int_{\RR} \psi (x) dx =0; \EE
such a function $\psi$ will be called an {\em admissible wavelet.} 
Let
\[ \forall a>0 ,\quad  b \in \RR, \qquad  \psi_{a,b } (x) = \psi \left ( \frac{x-b}{a}\right).
\]
If $f \in L^1_{loc} (\RR)$, the continuous wavelet transform of $f$ is 
\[ C_f(a, b ) = \frac{1}{a} \int_{\RR}  f(x) \psi_{a,b } (x)  dx. \]

In order to obtain sharp results, we need to extend the notion of $T^p_\al $ regularity to general moduli of continuity. We start by defining the   possible candidates: A function  $\theta: \;  \RR^+ \rightarrow \RR^+  $ satisfies  hypothesis ${ \mathcal H}$ if
\[ 
({ \mathcal H}) \quad  \left\{ \begin{array}{l}  \theta (0 ) =0 , \\  \theta    \mbox{ is  continuous  and non-decreasing {in a neighborhood of $0$}.}  \end{array}\right. 
\]

\BD\label{defmodgen} Let $\theta$ be a function satisfying ${ \mathcal H}$ and $f \in L^p_{loc}( \RR )$;  $\theta $ is a $p$-modulus of continuity of $f$  at $x_0$ if  
 there exists a polynomial $P$  such that, for $\rho$ small enough,   
\BE\label{modcont}  
\left( \int_{x_0-\rho}^{x_0+\rho} | f(x) -P( x-x_0)|^p dx 
\right)^{1/p}
\leq    \theta (\rho ). 
\EE
\ED
Note that $ T^p_\al $ regularity corresponds to  $\theta (\rho) =   C \rho^{\al + 1/p} $. 

\BL\label{lem3}
Let $\psi$ be an admissible wavelet, $p \in [1, + \infty]$ and  $f \in L^p_{loc} (\RR)$;  let  $\theta$ be a $p$-modulus of continuity of $f$ at $x_0$ satisfying 
\[ \exists C>0 , \; \forall \rho\in (0,1], \qquad  \theta (\rho ) \geq  C\rho^{1+ 1/p}. 
\]
Then 
\[ 
 supp \; (\psi_{a, b })  \subset [x_0- \rho, x_0+ \rho ] \;\Longrightarrow \;  | C_f(a, b  ) | \leq   2^{1-1/p}  \;  \theta (\rho ) \frac{\rho^{1-1/p}}{ a}. 
\]
\EL

This result will be used  for $p=1$ in order to prove the pointwise irregularity of the Brjuno function. 

\begin{proof}
The growth condition on $\theta$ implies that we can restrict to polynomials $P$ of degree 0. Since $\psi$ has a vanishing first moment,
\[ 
\forall D \in \RR, \qquad C_f(a, b  ) = \frac{1}{a} \int_{\RR}  (f(x)-D ) \psi_{a,b } (x)  dx. 
\]
Using \eqref{hyppsi}, we get  
\[ 
| C_f(a, b  ) |  \leq  \displaystyle\frac{1}{a} \int_{x_0- \rho}^{x_0+ \rho}  |f(x)-D |  dx 
\leq     \displaystyle  2^{1-1/p}  \; \frac{\theta (\rho ) }{a}   \rho^{1-1/p}.
\]
\end{proof}
 
Applying \eqref{BMfond}  to $h$ and $h/2$, we obtain that, if  $ 0< h  < {2}/{3 q^2}$, then 
\BE \label{difprim2} 
\frac{1}{h } \int B(x) H \left ( \frac{ x-r}{h}\right) dx  = \frac{\log 2}{q}  + { \mathcal O}    \left(  qh  \log \left(\frac{1}{q^2 h}\right)  \right) ,  
\EE
where $H  = 1_{[0, 1/2]} -  1_{ [1/2, 1]  } $ is the Haar wavelet. Hence the following result holds.

\BL \label{lemquat} Let $r = {p}/{q} $ with $p \wedge q =1$. If $0<  h  < {2}/{3 q^2}$, then 
\BE \label{tranOndB} 
\left| C_B \left(h, \frac{p}{q}\right) - \frac{\log 2}{q}  \right| \leq  \tilde{C}    qh  \log \left(\frac{1}{q^2 h}\right),  
\EE
where the wavelet used is the Haar wavelet and  the constant $ \tilde{C}$  is independent of $p,$ $q$ and $h$.
\EL 
We now introduce a notion of uniform irregularity   associated with moduli of continuity for $p =1$.
\BD 
Let $\theta$ be a   function satisfying ${ \mathcal H}$.  A function $f \in L^1_{loc} (\RR)$ is uniformly $\theta$-irregular if 
\[
\; \forall x_0,  \; \forall 
P, \;  \exists \rho_n \rightarrow 0 :  \quad 
\int_{x_0- \rho_n}^{x_0+ \rho_n}| f(x) -P (x-x_0) | \; dx  \geq  \theta (\rho_n)  . 
\]
\ED 
  
\BP \label{propirreg}  There exists $A >0$ such that  $B$  is uniformly $\theta$-irregular for   $\theta (\rho ) =  A \rho^{3/2}$; and this result is optimal (i.e. $\theta (\rho ) $ cannot be replaced by a $o(  \rho^{3/2})$).  
\EP  

The optimality of Proposition \ref{propirreg} will be proved in Section \ref{secbad} by considering  badly approximable numbers. 
  
\begin{proof} Let $x_0 \in \RR$. First note that, if $x_0 \in \QQ$, then the result follows from \eqref{irregrat}. If $x_0 \notin \QQ$, we apply \eqref{tranOndB}  to the sequence  
\[ 
r_n = \frac{p_n}{q_n}
\]  of convergents of $x_0$. We  now pick 
$h_n = {\eps }/{q_n^2 }$, where $\eps$ is positive and such that $\tilde{C} \eps \log (1/\eps ) \leq 1/4$  (where $ \tilde{C} $ is the constant in Lemma \ref{lemquat}). 
It follows that 
\[ C_B (h_n, r_n ) \geq \frac{1 }{4 q_n }.
\]
We now  apply Lemma \ref{lem3} with $a= h_n$, $b=r_n$ and
$ \rho_n =  \left| x_0 -r_n \right|  + h_n$; if $\theta$ is a 1-modulus of continuity at $x_0$, then 
\[
|C_B (h_n, r_n ) | \leq  \frac{\theta (\rho_n)}{ h_n },
\]
which implies that 
\[
\frac{1 }{4 q_n }  \leq \frac{\theta (\rho_n) }{ h_n } .
\]
Using that $\rho_n \leq 2/q_n^2 $ and $\theta $ is increasing, if follows that 
\[
\frac{\eps}{4 q_n^3}  \leq {\theta (\rho_n) }  \leq \theta (2/q_n^2 )= \frac{2A\sqrt{2}}{q_n^3}, 
\]
hence a contradiction if $A$ is small enough.
\end{proof}

Proposition \ref{propirreg} implies that the $1$-exponent  satisfies $ \forall x \in \RR$, $ h^1_B (x) \leq  {1}/{2}; $
thus we can assume in the following that  the polynomial in \eqref{modcont} boils down to  a constant  which has to be $B(x_0)$ as $x_0$ is a Brjuno number (recall that Brjuno numbers are Lebesgue points). 

Let us now check that the same argument as in the proof of Proposition \ref{propirreg}  yields  an irregularity result at points $x_0$ for which $\tau (x_0) >2$. Recall  
that an irrational  point $x_0$ is $\tau$-{\em well approximable}  if 
\[ \left| \x
-r_n \right| \leq \frac{1}{q_n^{\tau}}, 
\]
for infinitely many $n$s. 

\BL\label{lem:wapproximable}
Let $\tau>2$. If $x_0$ is $\tau$-well approximable, then 
$\theta(\rho) = \tfrac18\rho^{1+1/\tau}$ is not a modulus of continuity of $B$ at $x_0$ (so that 
$ 
h^1_B(x_0)\le {1}/{\tau}$).
\EL
\begin{proof}
Assume that $\theta(\rho) =\tfrac18 \rho^{1+1/\tau}$ is a modulus of continuity of $B$ at $x_0$. We pick $h_n=1/q_n^\tau$.    
As above 
$ C_B(1/q_n^\tau,r_n) \sim \log(2)/q_n$ when $n\to +\infty$. 
We apply Lemma \ref{lem3} with $a= h_n $, $b=r_n$ and 
$\rho_n = |x_0-r_n|+h_n$ so that $\rho_n \le 2/q_n^\tau$. We get 
$1/2 \le  2^{1+1/\tau}/8,$
hence a contradiction.  
\end{proof}
 
Recall that an irrational  number $x_0$ is Diophantine if $\tau (x_0) < \infty$; Liouville numbers are the irrational numbers that are not Diophantine.
It follows from  Lemma \ref{lem:wapproximable} that  {\em the $1$-exponent  of the Brjuno function vanishes at   
Liouville numbers.}
Moreover if $x_0$ is such that $\tau(x_0)>2$, then Lemma \ref{lem:wapproximable} gives  
\[
h^1_B(x_0) \le  \frac{1}{\tau(x_0)}. 
\]

\subsection{Pointwise regularity of the Brjuno function }\label{sec:regularity-B}
We now prove regularity for $B$ at  Diophantine numbers of $X = (0,1) \setminus \Q$. 

We begin by recalling  basic points about the continued fraction expansion of irrational numbers.
First, the Gauss map  $A$: $X \rightarrow X$ is defined by 
 \[ 
 A(x)  = \left\{   \frac{1}{x}  \right\}  
 \]
where $\left\{   x  \right\}= x -\lfloor x\rfloor $ denotes the fractional part of $x$ and $\lfloor x\rfloor $ its integer part of $x$. For $n\in \N$, we denote by $A^n$ the $n$-th iterate of $A$. 
\[ \mbox{If} \;   x\in X \;  \mbox{and }   n\ge 0,   \quad 
A^n(x) = \frac{p_{n-1}(x) - x \; q_{n-1}(x)}{q_n(x) x -p_n(x)} \;  \mbox{and}  \; 
 \lfloor A^n(x)\rfloor=a_n(x),
\] 
see e.g. \cite{billingsley-1965} p. 40-41.
 
We will denote by 
$
\cgot[b_1,\ldots,b_k]$ the open sub-interval of (0,1) with endpoints
$[0;b_1,\ldots,b_k]$ and $[0;b_1,\ldots,b_{k-1},b_k+1]$.  These intervals are called {\em cylinders of order} $k$. 
Note that in a cylinder of order $k$,  $A^k$ is continuous, and  for all $j\le k$ the functions $a_j$, $p_j$ and $q_j$ are constant. 
For $x\in X$,  let
\[
\beta_n(x)= |xq_n(x)-p_n(x)| 
\]
and  
\[  \gamma_n (x) =  \beta_{n-1}(x) \log \left( \frac{1}{A^n (x) }\right)
\]
so that 
\[ B(x) =  \sum_{n=0}^\infty\; \gamma_n (x) . 
\]
We have  
\[
\beta_n(x)= \frac{1}{q_{k+1}(x)+A^{k+1}(x) q_k(x)}, 
\] 
from which we get the  well-known bounds 
(see e.g \cite{billingsley-1965} p.42)
\BE
\label{enca-betan} 
\frac{1}{2 q_{n+1}(x)} \le \beta_n(x) \le  \frac{1}{q_{n+1}(x)}. 
\EE
It follows that 
\BE\label{taun-qn}
q_{n+1}(x) \le q_n(x)^{\tau_n(x)-1}. 
\EE
Let us also recall (see Proposition 1 of \cite{BaMa}) that for $k\ge 1$, 
\BE
\label{majo-gamma}
\frac{\log(q_{k+1}(x))}{q_k(x)}- \frac{\log(2 q_k(x))}{q_k(x)}\le \gamma_k(x) 
\le \frac{\log(q_{k+1}(x))}{q_k(x)}.
\EE

Let $x_0\in X$. In the sequel, $a_k,p_k,q_k$ denote the value at $x_0$ of the functions $a_k,p_k,q_k$. We need to estimate the integrals
\[
\int_I \gamma_k(t)\,  dt,  
\]
where 
\BE\label{inti} 
I = (x_0-\rho/2;x_0+\rho/2) \quad \mbox{ with} \quad \rho>0.  
\EE
These estimates will depend on an integer $K$ which is defined as follows: 
$K\; (=K(I))$ is  the largest integer   such that 
\BE\label{eq:def-K}
 I\subseteq \cgot[a_1,\ldots,a_K]. 
 \EE
We also denote by $\{F_k\}_{k\ge0}$ the sequence of Fibonacci numbers (i.e. $F_0=F_1=1$, $F_{n+2}=F_{n+1}+F_n$).   

\begin{lem}\label{esti-integrals}
Let $x_0\in X$, let $\rho$ be such that $0<\rho<e^{-2}$ with  $x_0-\rho/2$ and $x_0+\rho/2$ irrational. Let $I$ be the interval  be given by \eqref{inti} and $K$ the integer defined by \eqref{eq:def-K}. There exists an absolute constant $C >0$ such that for $K\ge 1$, 
\begin{align}
\forall k<K,\quad  &  \int_I\big| \gamma_k(x)-\gamma_k(x_0) \big| dx \leq C  q_{k+1} \rho^2, \label{int-kpetit}\\
  & \int_I \big| \gamma_K(x)-\gamma_K(x_0) \big| dx  \leq C q_{K+1} \rho^2\log(q_{K+1}\big),
   \label{int-kventral}\\
\forall k>K, \quad &   \int_I \gamma_k(x)dx \leq C  \frac{\rho}{F_{k-K}} \Big( \frac{\log(1/\rho)}{q_{K+1}} +\rho^{1/2}\Big). \label{int-kgrand}
\end{align}
\end{lem}

\begin{proof}
The bound \eqref{int-kpetit} is exactly Proposition 7 of \cite{BaMa}.  Propositions 9 and 10 of the same paper give bounds for the integrals $ \int_I \gamma_k(x)dx$ for $k\ge K$, however  \eqref{int-kventral} and  \eqref{int-kgrand} cannot be directly derived from them, but  will be a consequence of their 
proofs. Let us  recall  the  notations of \cite{BaMa}.
The endpoints of the interval $\cgot[a_1,\ldots,a_K]$ are 
\[
\frac{p_K}{q_K} \quad \textrm{ and } \quad \frac{p_K+p_{K-1}}{q_K+q_{K-1}}. 
\] 
Up to some subset of  $\Q$, $\cgot[a_1,\ldots,a_K]$ is the union of the cylinders 
$\cgot[a_1,\ldots,a_K,n]$ over $n\ge1$. Any element $x$ of $\cgot$ has a unique representation 
\[
x = \frac{sp_K+p_{K-1}}{sq_K+q_{K-1}}\quad  \textrm{ with }\quad  s \in]1;+\infty[. 
\]  
We set 
\[
x_0 +(-1)^{K}\rho/2 = \frac{up_K+p_{K-1}}{uq_K+q_{K-1}} \textrm{ and }x_0 +(-1)^{K-1}\rho/2 = \frac{vp_K+p_{K-1}}{vq_K+q_{K-1}},
\]
and 
\BE\label{eq:def-mn}
m = [u] \textrm{ and } n=[v],
\EE
 so that 
$
1\le m \le a_{K+1}\le n$. By maximality of $K$, we have $n>m$.
Inequality (40) of \cite{BaMa} gives  
\BE
\label{mino-h}
\rho\ge \frac{v-u}{6q_K^2mn}. 
\EE
Let us now prove  \eqref{int-kventral}. We distinguish two cases. 

First, suppose that $n\ge 2m+1$. Then  
 $v-u\ge (n-m)/2$ and we obtain from \eqref{mino-h} that 
\BE\label{mino-hcas1}
\rho\ge \frac{1}{24mq_K^2} \ge \frac{1}{24a_{K+1}q_K^2}\ge \frac{1}{24q_Kq_{K+1}}. 
\EE
Proposition 9 of \cite{BaMa} and  \eqref{majo-gamma} give
\BE\label{eq:majoBaMa-P9}
 \int_I  \gamma_K(x) \le C \rho \frac{\log(q_{K+1})}{q_K}.   
\EE
From this we deduce 
\begin{align*}
 \int_I \big| \gamma_K(x)-\gamma_K(x_0) \big| dx
& \le \int_I \gamma_K(x)dx + \rho \gamma_K(x_0) \\
& \leq C  \frac{\rho}{q_{K}}\log(q_{K+1}) \leq C \rho^2 q_{K+1} \log(q_{K+1}),
\end{align*}
where the last inequality comes from \eqref{mino-hcas1}.

Suppose now that $m\le n \le 2m$. If $x\in I$, the derivative of  $\gamma_K$ satisfies
\[
\gamma_K'(x)= (-1)^{K-1}\Big( q_{K-1}(x) \log(1/A^K(x))+\beta_K(x)^{-1}\Big),
\]
so that 
\[
|\gamma_K'(x)| \leq C q_{K+1}(x). 
\]
For $x \in I $, there exists $m\le \ell \le n$ such that $x\in\cgot[a_1,\ldots,a_K, \ell]$ which yields 
\[
q_{K+1}(x) = \ell q_K + q_{K-1} \le n q_{K}+q_{K-1} \le 2 a_{K+1}q_K+q_{K-1}
\le 2 q_{K+1}. 
\]
By the  mean-value theorem, 
\[
\int_I \big| \gamma_K(x)-\gamma_K(x_0) \big| dx
\leq C q_{K+1} \rho^2,
\]
and the case $k=K$ is settled. 
\smallskip 

It remains to consider the case $k>K$. Let 
\BE\label{eq:def-E} 
E=n-m+1.
\EE If $E=2$,  inequality (43) of \cite{BaMa} gives 
\BE\label{eq:majoBaMa-P10a}
\int_I \gamma_k(x)dx  \leq  \frac{2e }{q_{K+1} F_{k-K}}\rho\log(1/\rho). 
\EE
If $E\ge 3$, then $v-u\ge (n-m)/2$ so that \eqref{mino-h} gives  
$
\rho \ge \frac{n-m}{12 q_K^2 mn}.
$
Using 
\BE\label{eq:majoBaMa-P10b}
\int_I \gamma_k(x)dx
\le 6 \frac{n-m}{q_K^3 F_{k-K} m^2n} 
\EE
(see p. 213 of \cite{BaMa}),  it follows that
\[
\int_I \gamma_k(x)dx
\le \frac{(12 \rho)^{3/2}}{F_{k-K}} \Big( \frac{n}{(n-m)m} \Big)^{1/2}
\le \frac{(12 \rho)^{3/2}}{F_{k-K}}. 
\]
\end{proof}

We are now able to prove the following result.

\begin{prp} \label{propder}
 Let $x_0$ be a Diophantine  number, and $\eps>0$.  There exists $C=C(x_0)>0$ and $\rho_0=\rho_0(x_0,\eps)>0$ such that, if  $0<\rho<\rho_0$, then 
 \BE\label{B-taueps}
\frac{1}{\rho} \int_{x_0-\rho/2}^{x_0+\rho/2}|B(x_0)-B(x)| dx \leq C\rho^{1/(\tau(x_0)+\eps)}\log(1/\rho).
 \EE
\end{prp}

From \eqref{B-taueps} we deduce that,  if $x_0$ is  Diophantine, then  for every $\eps$ such that $0<\eps<1/2$, $h^1_B(x_0)\ge 1/(\tau(x_0)+\eps) $, and consequently $h^1_B(x_0)\ge \frac{1}{\tau(x_0)}$ (as $\tau(x_0)\ge2)$) which ends the proof of Theorem \ref{thm:SpecB}. We now prove Proposition \ref{propder}.

\begin{proof}
As the set of irrational numbers is dense in $\R$, we may  assume that $x_0\pm \rho/2$ are both irrational. 
Let $\eps>0$. There exists an integer $K_0=K_0(x_0,\eps)$ such that 
 \BE
 \label{majo-tauK}
\forall K\ge K_0, \qquad \tau_K(x_0)\le \tau(x_0)+\eps.
 \EE
 We will note $\tau_k(x_0)=\tau_k$ and $\tau(x_0)=\tau$. 
Following \cite{BaMa},  $\delta_k=\delta_k(x_0)$ will denote the distance from $x_0$ to 
the endpoints of $\cgot[a_1,\ldots,a_k]$, i.e.
\[
\delta_k = \min\Big( \Big| x_0-\frac{p_k}{q_k}\Big|, \Big|x_0-\frac{p_k+p_{k-1}}{q_k+q_{k-1}} \Big| \Big).  
\]
Proposition 4 of \cite{BaMa} gives 
\BE
\label{enca-distance}
\delta_k \le \frac{1}{q_kq_{k+1}} \quad \textrm{ and } \quad 
\delta_k\ge \
\begin{cases}
 \frac{1}{2q_{k+1}q_{k+2}} & \textrm{ if }a_{k+1}=1,\\
 \frac{1}{2q_{k}q_{k+1}} & \textrm{ if }a_{k+1}\ge 2.
\end{cases}
\EE
Let $K= K(x_0,\rho)$ be the largest integer such that $I = (x_0-\rho/2;x_0+\rho/2)$ is included in $\cgot[a_1,\ldots,a_K]$.  We have 
 \[
\frac{1}{2q_{K+2}q_{K+3}}\le \delta_{K+1}< \rho/2
\]
so that $K\to +\infty$ when $\rho\to 0$. Let $0<\rho_0<e^{-2}$ be such that for all  $0<\rho<\rho_0$,
$K\ge \max(K_0,1)$ and let us evaluate for $\rho<\rho_0$, 
\[
 \int_I |B(x_0)-B(x)| dx \le 
 \sum_{k\le K} \int_I |\gamma_k(x_0)-\gamma_k(x)|dx
 + \rho\sum_{k>K} \gamma_k(x_0) + \sum_{k>K}\int_I \gamma_k(x)dx.  
\]
Using  \eqref{int-kpetit} and \eqref{int-kventral}, and since  the sequence $\{q_k\}_{k\ge0}$ grows (at least) exponentially, it follows that   
\begin{align*}
 \sum_{k\le K} \int_I |\gamma_k(x_0)-\gamma_k(x)|dx&
 \leq C \rho^2\Big( \sum_{k<K} q_{k+1} + q_{K+1}\log(q_{K+1})\Big)\\ 
 &\leq C \rho^2 q_{K+1}\log(q_{K+1}). 
\end{align*} 
Now, using \eqref{int-kgrand}, we get 
\begin{align*}
 \sum_{k>K}\int_I \gamma_k(x)dx&
 \leq C \rho \Big( \frac{\log(1/\rho)}{q_{K+1}}+ \rho^{1/2}\Big)
 \sum_{k>K}\frac{1}{F_{k-K}}\\
 &\leq C \rho \Big( \frac{\log(1/\rho)}{q_{K+1}}+ \rho^{1/2}\Big).
\end{align*}
(because  the sequence $\{F_k\}_{k\ge0}$ of Fibonacci numbers grows exponentially). 

Since $x_0$ is Diophantine,  $\tau(x_0) < \infty$; therefore the sequence $(\tau_k)_{k\ge0}$ is bounded. Using \eqref{taun-qn} and \eqref{majo-gamma}, we get for $k>K$, 
\[  |   \gamma_k (x_0) |   \le   \frac{\log(q_{k+1})}{ q_k} \leq  \frac{\log(q_{k}^{\tau_k-1})}{q_k}
\leq C \frac{\log(q_k)}{q_k},
\] 
where $C$ depends on $x_0$. 
We deduce from this and again from the exponential growth of $\{q_k\}_{k\ge0}$ that 
\[
\sum_{k>K} \gamma_k(x_0) \leq C\, \frac{\log(q_{K+1})}{q_{K+1}}.
\]
Collecting these estimates we get  
\BE\label{prefinal}
 \int_I |B(x_0)-B(x)| dx
 \leq C \rho \Big( \rho\, q_{K+1} \log(q_{K+1}) + \frac{\log(1/\rho)}{q_{K+1}}+ \rho^{1/2} \Big).
\EE
According to \eqref{taun-qn} 
\[ 
q_{K+1} \le q_K^{\tau_{K}-1}= \Big| x-\frac{p_K}{q_K} \Big|^{(1-\tau_K)/\tau_K}
\le \rho^{-1+1/\tau_K},
\]
from which we deduce that
$
\log(q_{K+1}) \leq C \log(1/\rho).
$
If $a_{K+2}\ge 2$, according to \eqref{enca-betan}  and  \eqref{enca-distance},
\[
\frac{1}{q_{K+1}^{\tau_{K+1}}}=
\Big|x- \frac{p_{K+1}}{q_{K+1}} \Big|\le \frac{1}{q_{K+1}q_{K+2}} 
\le 2 \delta_{K+1} <\rho, 
\]
and if $a_{K+2}=1$ we get in the same way 
\[
\frac{1}{q_{K+1}} \le 
 \frac{2}{q_{K+2}} \le 2 \rho^{1/\tau_{K+2}}. 
\]
Inserting these  estimates in \eqref{prefinal}, 
we finally get 
\[
 \int_I |B(x_0)-B(x)| dx \leq C 
 \rho\Big(\rho^{1/\tau_{K}}+\rho^{1/\tau_{K+1}}+\rho^{1/\tau_{K+2}}\Big)\log(1/\rho) ,   
\]
(note that, since $\tau_K  \geq 2$, the term $\rho^{1/2}$ in \eqref{prefinal} is not needed)
and the conclusion follows  from \eqref{majo-tauK}.
\end{proof}

We now turn to the proof of Corollary \ref{corspec}. It follows  from  a precised version of Jarnik Theorem (which, initially, yields the Hausdorff dimensions of the sets of points with  a given irrationality exponent). In order to state it, we need to recall the following notion of modified Hausdorff measure.

\BD  \label{defmeshaus} Let   $A\subset \RR$. 
If $\varepsilon>0$  and  $\delta \in [0,1]$,  we denote 
\[ M^{\delta,\gamma}_{\varepsilon} =\inf_R \;  \left( \sum_{  i} | A_i |^\delta | \log ( | A_i | ) |^\gamma \right) ,\]
where $R$ is an  $\eps$-covering of   $A$, i.e. a covering of  $A$  by bounded sets  $\{ A_i\}_{i \in
\N}$ 
 of diameters  $| A_i | \leq \varepsilon$
(the infimum is therefore taken on all  $\eps$-coverings).
For any $\delta \in [0,1]$ and $\gamma \in \RR$, the 
$(\delta, \gamma )$-dimensional outer  Hausdorff measure of 
$A$ is 
\[ mes^{\delta,\gamma} (A) = \displaystyle\lim_{\eps\rightarrow 0}  M^{\delta,\gamma}_{\eps}. \]
\ED

\BP \label{propjar}
Let $ a, b \in \RR$ such that $a < b$ and 
\[ E_{\tau}=\left\{  x\in [a,b] : \; \left|x -\frac{p}{q}\right| \leq \frac{1}{q^{\tau}} \;
\mbox{ for
 infinitely many couples $(p,q)$ } \right\} ;\]
 then 
 \begin{eqnarray}  \label{jar1}
\dim ( E_{\tau} ) = 2/\tau ,
 \\  \label{jar2}
 mes^{2/\tau, 2} ( E_{\tau}  )  >0. 
\end{eqnarray} 
%
\EP

 Note that the upper bound in (\ref{jar1}) follows immediately using the natural covering by the intervals $[ \frac{p}{q}-\frac{1}{q^{\tau}}, \frac{p}{q}+\frac{1}{q^{\tau}}]$, and (\ref{jar2}) implies the lower bound, so that the only result that requires a proof is (\ref{jar2}). A direct and  elementary way in order to prove (\ref{jar2}) is to follow step by step the proof of Theorems 10.3 and 10.4 of   \cite{Falconer:2003oj}, and check that it actually yields not only a lower bound for the dimension of $E_{\tau}$, but  the more precise result given by (\ref{jar2}). A more conceptual proof consists in noticing that the sets $E_{\tau}$ are {\em limsup sets}, and that, by Dirichlet's theorem, $E_2= \RR$; so that Proposition \ref{propjar} actually  is a particular case of the standard ubiquity techniques, see \cite{Beresnevich:2006ve} and references therein, or Theorem 2 of \cite{Lev}.   
 
 Let us now check how Corollary \ref{corspec} follows from  Theorem \ref{thm:SpecB} and  Proposition \ref{propjar}. Let $a <b $ be given and let 
 \[ F_t =  \{ x\in [a,b]:\;  \tau (x) = t \} . \]
 We need to prove that 
 \BE \forall t \in [2, +\infty], \qquad  \dim  \left( F_t \right) = \frac{2}{t}. 
 \EE 
 
Clearly,
\[ F_{t}= \displaystyle\bigcap_{\tau  < t} E_{\tau }-\displaystyle\bigcup_{\tau  >t}E_{\tau }.\]
It follows that $\forall \tau  < t$, $F_{t} \subset E_{\tau }$, and Proposition \ref{propjar}  implies that $\dim  \left( F_t \right) \leq 2/t$. 
In order to obtain the lower bound, we will prove that $mes^{2/t, 2} ( F_{t}  )  >0$. Indeed, $F_{t} $ contains the set
\[ G_{t}= \ E_{t }-\displaystyle\bigcup_{\tau  >t}E_{\tau }; \]
since the sequence $E_{\tau }$ is decreasing, the union can be rewritten as a countable union of sets, which, by (\ref{jar1}), all have a vanishing 
$mes^{2/t, 2}$ Hausdorff measure, so that $mes^{2/t, 2} (G_{t} )= mes^{2/t, 2} (E_{t})  >0$. \\

\section{Badly approximable numbers}

\label{secbad}

 Theorem \ref{thm:SpecB} can be interpreted as stating   that the slower the sequence $q_n$ increases, the smoother ${B}$ is at $x_0$. 
We  now   prove that, indeed, the points for which the sequence  $q_n$  grows as slowly as possible are the ones where $B$ is the smoothest. 

An irrational number $x_0$ is {\em badly approximable}
if the sequence of $\{a_k\}_{k\ge 0}$ is bounded,  or, equivalently, if 
\[ \exists C >0, \quad \forall p, q \neq 0,  \qquad 
\Big| x_0 - \frac pq   \Big| \ge \frac{C}{q^2}.
\]
It  follows that 
$\tau(x_0)=2$; thus  we already know that 
$
h_B^1(x_0) = {1}/{2}$. 
We now sharpen  this result. Recall that the definition of the modulus of continuity is given by
\eqref{modcont}. 

\BP \label{propbad}
A point  $x_0\in (0,1)$ is badly approximable if and only if  there exists $C>0$ such that $\theta(\rho)=C \rho^{3/2} $ is a modulus of continuity of $B$ at $x_0$. 
\EP 

A consequence  is the optimality of  Proposition \ref{propirreg} (up to the multiplicative constant): Badly approximable numbers have the smallest possible modulus of continuity.   

\begin{proof}  First note that  a function which is a $o(\rho^{3/2} )$ cannot be a modulus of continuity at badly approximable numbers, as a consequence of  Proposition~\ref{propirreg}. We now prove that, for $C$ large enough, $C \rho^{3/2}  $ is a modulus of continuity at such a number. In this proof, the values of $C$ may change from one line to the next, but only depend on $x_0$. 
 We will use  the same notations ($I$, $K$, $E$, $\delta_K$) as in the proofs of Lemma \ref{esti-integrals} and Proposition \ref{propder}. 
First note that 
\BE\label{eq:majoh}
\rho \le \Big|x_0 - \frac{p_K}{q_K} \Big|\le \frac{1}{q_K^2},
\EE
and also, as $x_0$ is badly approximable,  
\BE\label{eq:minoh}  \exists C : \quad
\frac{\rho}{2}> \delta_{K+1} \ge \frac{1}{2q_{K+2}q_{K+3}}
\ge \frac{C}{q_K^2}.  
\EE
According to \eqref{int-kpetit} and  \eqref{eq:majoh}, 

\BE\label{eq:majo0-Ba}
\sum_{k< K} \int_I | \gamma_k(x)-\gamma_k(x_0) |dx 
\le C  \rho^2 q_{K}\le C  \rho^{3/2}. 
\EE 

Let $k>K$. The proof of Proposition 10 of \cite{BaMa} p. 213 contains the following inequality for $E\ge 3$, which actually remains true for $E\ge 2$ : 
\[
\int_I \gamma_k(t)dt \le \frac{2}{q_K^3 F_{k-K}}\sum_{m\le  \ell \le n} \frac{1}{\ell^3}. 
\]
This inequality and \eqref{eq:minoh} directly imply 
\BE\label{eq:majo1-Ba}
\sum_{k>K}\int_I  \gamma_k(x)dx \le C \rho^{3/2}. 
\EE

As 
$
1/A^k(x_0)= a_{k+1}(x_0)+A^{k+1}(x_0), 
$
there exists  $C>0$
such that for all $k\in\N$, $\log(1/A^k(x_0)) \le C$; using the exponential growth of the  $(q_k)_{k\ge 0}$, we get
\BE\label{eq:majo2-Ba}
\rho \sum_{k\ge K}\gamma_k(x_0) 
\le C \rho \sum_{k\ge K} \frac{1}{q_k}
\le C \frac{\rho}{q_{K}} \le C \rho^{3/2}.    
\EE 
\smallskip 

To treat $
\int_I \gamma_K(x)dx$, 
 we use  (39) of \cite{BaMa} :
\BE\label{eq:majo3-Ba}
\int_I \gamma_K(x)dx \le \frac{1}{q_K^3}
\int_{A^K(I)} \log(1/u) du
\le  \frac{1}{q_K^3}
\int_0^{|A^K(I)|} \log(1/u) du \le C\rho^{3/2}, 
\EE
for $u \mapsto \log(1/u)$ is decreasing on $(0,1]$ and $\int_0^1 \log(1/u) du < \infty$. 
Collecting the estimates \eqref{eq:majo0-Ba}, \eqref{eq:majo1-Ba}, \eqref{eq:majo2-Ba} and \eqref{eq:majo3-Ba} 
we obtain that $C \rho^{3/2}$ is a modulus of continuity at $x_0$. 
\medskip 

We now prove that badly approximable points are the only one for which the modulus of continuity is equivalent to $\rho^{3/2}$.  Let $ h_{n} = \left| x_0 - {p_{n}}/{q_{n}}\right| ;  $ 
 $x_0$ is not badly approximable if and only if there exists a subsequence $n(m)$ such that for $m\to \infty$, 
\BE \label{petitoo}  h_{n(m)} =o \left(  \frac{1}{(q_{n(m)})^2}\right) . 
\EE
The proof then follows the one of Proposition \ref{propirreg}: On one hand, \eqref{tranOndB} implies that
\[C_B \left(h_{n(m)},  \frac{p_{n(m)}}{q_{n(m)}} \right) \geq  \frac{1}{4 q_{n(m)}} ;
\]
on other hand, applying Lemma \ref{lem3} with $\rho_{n(m)} = 2 h_{n(m)} $, we obtain that, if  $C\rho^{3/2}$ is a modulus of continuity at $x_0$, then 
\[ \frac{1}{4 q_{n(m)}} \leq \frac{C (\rho_{n(m)})^{3/2}}{h_{n(m)}} \leq C   2^{3/2} (h_{n(m)})^{1/2}, 
 \]
 which contradicts \eqref{petitoo}. 
\end{proof}

\section{Additional pointwise regularity results}\label{sec:addit} 

We start by showing why the pointwise exponent used for positive measures  is not relevant for $B$. Recall that, if   
$\mu$  is  a  positive Radon  measure  defined on $ \R$,  
The  {\em   local dimension} of $\mu$ at  $x_0$ is 
\[
\underline{\dim}_{loc} (\mu, x_0)  = \liminf_{\rho\to 0^+} \frac{\log \mu([x_0-\rho, x_0+\rho ])}{\log \, \rho}. 
\]

The local dimension is  well defined for the Brjuno function; however, it  does not allow to capture  possible changes in its pointwise  regularity.
Indeed, let us check that 
it is constant. 

First, clearly,  $\exists C >0$ such that $ \forall x \in \R$, $B(x) \geq C$ because   there is no cancellation in the series  \eqref{B2}; 
so that $\forall x, $ $\underline{\dim}_{loc} (\mu, x) \leq 1$. 
On  other hand, since  $B \in BMO$, it follows immediately from \eqref{JohnNin} 
that  
 $\forall x, $ $\underline{\dim}_{loc} (\mu, x) \geq 1$.  
 
 A  drawback of  using the local dimension   is that, in contradistinction with the H\"older exponent, two measures $\mu$ and $\nu$ differing by a constant may have different exponents. 
 Therefore this exponent often takes  the value 1, because the definition does not include  (as in the H\"older case) the substraction of an appropriate polynomial. This explains why the $p$-exponent, which allows for this  substraction, is better fitted to measure  variations of  regularity of $B$.

\subsection{H\"older regularity of the primitive of $B$}

\label{secregprimit} 
The proof of Theorem \ref{thm:SpecB}  strongly uses \eqref{BMfond}, which estimates increments of  the primitive of $B$; therefore a natural question is to wonder if it  can yield its {\em  H\"older exponent}. 

\begin{df}\label{df:holponc} Let $f:\R\to\R$ be a locally bounded function, $x_{0}\in\R$ and $\alpha \geq 0$. The function $f$ belongs to $C^{\alpha}(x_0)$ if there exist $C>0$ and a polynomial $P$ of degree less than $\alpha$ such that, for $\rho$ small enough,
\BE \label{defreghol}
\supess_{ | x-x_0| \leq \rho} |f(x) -P(x-x_0)|\leq C \rho^{\alpha}. 
\EE
The {H\"older exponent} of $f$ at $x_0$ is 
\[
h_{f}(x_0)=\sup\{ \alpha\geq 0 \:|\: f \in C^{\alpha}(x_0)\} .
\]
\end{df}

We denote by $\TB$ a primitive of $B$. A lower bound for  $h_{\TB}$ is a consequence of   the following classical result.

\BL\label{lem:mino-exp-primitive} Let  $f \in L^1_{loc} (\RR)$, and denote by $F$ a primitive of $f$. Then \BE \label{resprimitgen} \forall x_0\in \RR, \qquad  h_F (x_0) \geq h^1_f (x_0) +1. \EE
\EL

\begin{proof} We recall the proof for the sake of completeness.  Suppose that $f \in T^1_{\al} (x_0)$, 
let $P$ be the polynomial given by (\ref{deux}), and denote by $Q$ the primitive of $P$ that vanishes at $x_0$. The primitive 
$F (x) = \int_{x_0}^x f(t) dt $ satisfies 
\begin{align*}
\left| F(x)-Q(x) \right| &=  \left| \int_{x_0}^x (f(t) -P(t)) dt \right|\\& \leq\int_{x_0}^x   \left| f(t) -P(t)  \right| dt \leq  C |x-x_0|^{\al +1}, 
\end{align*}
so that $F \in C^{\al +1} (x_0)$.  
\end{proof} 

Note that, in general, equality does not hold in (\ref{resprimitgen}), as shown by the functions $|x|^\al \sin ( |x|^{-\beta})$ for $\al >-1$ and $ \be >0$; we will now check that equality holds everywhere in the case of the Brjuno function. 

\BP\label{propregprimit}
If $x_0 \in \QQ$, then $h_{\TB} (x_0) = 1$. Otherwise, 
\[ h_{\TB} (x_0) = 1+  \frac{1}{\tau (x_0)}.  
\] 
\EP

In order to prove this result, we will need an irregularity criterium   based on finite differences. We note 
\BE \label{diffin} \Delta_2 f(x,h) =2 f\left(x+\frac{h}{2}\right) - f(x+h)-f(x). \EE 

\BL \label{lemdiffde}
Let $f: \RR \rightarrow \RR$ be a continuous function;  let $\al <2$ and $\gamma \geq 0$. Let $x_0 \in \RR$, and assume that there exist  $ \rho_n >0$, 
$h_n$,  and $r_n $ such that 
\[ r_n \in [x_0- \rho_n, x_0 + \rho_n] , \quad  \rho_n \rightarrow 0,  \quad \mbox{ and} \quad \frac{\rho_n}{| \log \rho_n |^\gamma}\leq | h_n | \leq \rho_n . 
\]
\[ \mbox{ If } \quad |  \Delta_2 f(r_n,h_n)|  \geq |h_n|^\al  ,\qquad \mbox{then} \quad h_f (x_0 ) \leq \al. 
\]
\EL

\begin{proof} Clearly, if $f$ is continuous, then the $\supess$ in \eqref{defreghol}  can be replaced by a $\sup$.  Therefore, if $f \in C^\beta (x_0)$ for a $\beta < 2$,  then there exists a polynomial $P$ of degree at most 1 and $r >0$ such that 
\BE \label{gradosup}  \forall x \in [x_0- r, x_0 +r ] , \qquad f(x) = P(x-x_0) + O( | x-x_0|^\beta ). \EE
Using \eqref{gradosup}  for $x=r_n, \; r_n + h_n/2$ and $r_n + h_n$ in \eqref{diffin}, we get
\[ \Delta_2 f(r_n,h_n)  = O (\rho_n^\beta ) = O\big(|h_n|^\beta \big| \log |h_n| \big|^{\beta \gamma } \big) .
\]
Therefore, if $| \Delta_2 f(r_n,h_n) | \geq | h_n|^\al$, then $\forall \be >\al$, $  f \notin  C^\beta (x_0)$. 
\end{proof} 

Let us now prove Proposition \ref{propregprimit}.  The case $x_0\in\Q$ follows from  \eqref{BMfond}. If $x_0\notin \Q$,   \eqref{resprimitgen} and Theorem \ref{thm:SpecB} imply  that $h_{\TB} (x_0) \geq 1+ 1/\tau(x_0)$.  Note that \eqref{difprim2} can be rewritten
 \[ 
 \mbox{If}  \quad  | h | < \frac{2}{3 q^2}, \;  \mbox{then} \qquad  \frac{1}{h} \Delta_2 \TB (r,h)  =  \frac{\log 2}{q}  + { \mathcal O}    \left(  qh  \log \left(\frac{1}{q^2 |h|}\right)  \right) .  
\]
 Let now $\tau \geq 2$ and assume  that $x_0$ is $\tau$-well approximable.  We can assume, by extracting a subsequence if necessary, that
for all  $n\ge 1$, $\big|x_0-p_n/q_n\big| \le 1/q_n^{\tau}$,  and 
we pick for $r$ the sequence of convergents $r_n = p_n/q_n$,   $\rho_n = 1/q_n^\tau$, and $h_n =   1/q_n^\tau (\log q_n)^2$. We obtain that 
 \[ \Delta_2 \TB (r_n,h_n) = \frac{\log 2}{\tau^{2/\tau}}  \; (h_n )^{1+1/\tau} |\log(h_n)|^{2/\tau}\Big(1+o(1)\Big), 
 \]
and Lemma \ref{lemdiffde} implies that $h_{\TB} (x_0) \leq 1+ 1/\tau$. \hspace{4cm} $\square$

\subsection{The $p$-exponent  of $B$ for $p>1$}

\label{secpexps}

 Theorem \ref{thm:SpecB} can be extended to $p$-exponents in the following way :
\BE \label{pexpbrj}
\forall p\ge 1,\; \forall x_0 \in \RR,  \qquad 
 h_B^p(x_0)= h_B(x_0). 
\EE

We  outline the proof of this result. 
Using the John-Nirenberg inequality,
\[ \exists C>0 ,  \forall x,y :  \;\;  |x-y|\leq \frac{1}{2}, \quad { \Big| \int_x^y  B(t)^p   \; dt \Big|^{1/p}}  \leq C |x-y|^{1/p} \log \left( \frac{1}{|x-y| }\right),
\]
so that 
$ \forall x_0 \in (0,1)$, 
$h_B^p(x_0)\ge 0$. 
On other hand,  H\"older's inequality  implies that 
\BE\label{eq:decroissance-p-expo}
h_B^p(x_0) \le h_B^1(x_0).  
\EE

Hence it follows from Theorem 1 that if $x_0$ is a rational or a Liouville number, then  $h_B^p(x_0)=0$.  

Suppose now that $x_0$ is Diophantine. The results obtained in section \ref{sec:regularity-B} are based on the estimates \eqref{int-kpetit}, \eqref{eq:majoBaMa-P9}, \eqref{eq:majoBaMa-P10a} and \eqref{eq:majoBaMa-P10b} from \cite{BaMa}. They extend  as follows: 
Let $\rho$ be such that $0<\rho<e^{-2}$ with  $x_0-\rho/2$ and $x_0+\rho/2$ irrational, $I= (x_0-\rho/2;x_0+\rho/2)$, and $K$, $m$, $n$, $E$ the integers defined by \eqref{eq:def-K}, \eqref{eq:def-mn} and \eqref{eq:def-E}. There exists an absolute constant $C >0$ such that for $K\ge 1$,
\begin{align}
\forall k<K,\quad  & \left( \int_I \big| \gamma_k(x)-\gamma_k(x_0) \big|^p dx\right)^{1/p} \leq C  
q_{k+1} \rho^{1+1/p}, \label{int-kpetit-p}\\
  & \left(\int_I  \gamma_K(x)^p dx\right)^{1/p} \leq C\, \frac{\rho^{1/p} \log(q_{K+1})}{q_K},
   \label{int-kventral-p}\\
\forall k>K, \quad &  \left( \int_I \gamma_k(x)^pdx\right)^{1/p} \leq C \, \frac{\rho^{1/p}\log(1/\rho)}{F_{k-K} q_{K+1}},\textrm{ if } E=2,\label{int-kgrand-p1}\\
&   \left(\int_I \gamma_k(x)^p dx\right)^{1/p} \leq C\, \frac{(n-m)^{1/p}}{F_{k-K} q_K^{1+2/p} m^{1+1/p}n^{1/p}}\textrm{ if } E\ge 3.\label{int-kgrand-p2}
\end{align}
The proofs  follow the same arguments as in the  proofs of Propositions 7, 9, 10 of \cite{BaMa}.
Doing so, easy extensions of Proposition 2 and Lemma 2 of the same paper will be requested. As the method is exactly the same to get these, we give them without proofs  : 
\begin{itemize}
\item Let  $I\subseteq[0;1]$ be an interval of length $h\le e^{-p}$. We have  
\BE\label{eq:gene-majoI}
\forall k\in\N,\qquad  \int_I \gamma_k(x)^p dx \le e^p h\log^p(1/h).  
\EE
\item Let $m,n$ be integers such that $1\le m<n$. We have
\[
\sum_{m\le \ell \le n}\frac{1}{\ell^{p+2}} \le 3 \frac{n-m}{m^{p+1} n}. 
\]
\end{itemize}

Starting from \eqref{int-kpetit-p}, \eqref{int-kventral-p}, \eqref{int-kgrand-p1} and \eqref{int-kgrand-p2}, one obtains the following extension of Lemma \ref{esti-integrals} : Under the same hypothesis there exists an absolute constant $C>0$ such that for $p\ge 1$,  $K\ge 1$, 
\begin{align*}
\forall k<K,\quad  & \Big( \int_I\big| \gamma_k(x)-\gamma_k(x_0) \big|^p dx\Big)^{1/p} \leq C  q_{k+1} \rho^{1+1/p}, \\
  & \Big(\int_I \big| \gamma_K(x)-\gamma_K(x_0) \big|^p dx\Big)^{1/p}  \leq C q_{K+1} \rho^{1+1/p}\log(q_{K+1}),
\\
\forall k>K, \quad &   \Big(\int_I \gamma_k(x)^p dx\Big)^{1/p} \leq C  \frac{\rho^{1/p}}{F_{k-K}} \Big( \frac{\log(1/\rho)}{q_{K+1}} +\rho^{1/2}\Big) .
\end{align*} 
The following extension of Proposition \ref{propder}  follows   : 
If $x_0$ is a Diophantine  number and $\eps>0$,   there exists $C=C(x_0)>0$ and $\rho_0=\rho_0(x_0,\eps)>0$  such that  for $p \ge 1$, $0<\rho<\rho_0$, 
 \BE\label{p-expo-dio}
\left(\frac{1}{\rho} \int_{x_0-\rho/2}^{x_0+\rho/2}|B(x_0)-B(x)|^p dx\right)^{1/p} \leq C\rho^{1/(\tau(x_0)+\eps)}\log(1/\rho).
\EE
From \eqref{p-expo-dio} we infer the lower-bound 
$h_B^p(x_0) \ge 1/\tau(x_0).$
Combined with \eqref{eq:decroissance-p-expo} and Theorem \ref{thm:SpecB}, this yields $h_B^p(x_0) = 1/\tau(x_0).$ 

\section{Concluding remarks  }  \label{sec:conclu}
The present paper  raises the problem of determining if Theorem \ref{thm:SpecB} also applies for  variants of the Brjuno function.

First, $B$  is one example of a family  $B_{\al}$  introduced by J.-C. Yoccoz in \cite{Yo}, and further studied in \cite{LuMaNaNa,MaMoYo}: In the definition of $B$,   the usual continued fraction algorithm  is replaced by $\al$-continued fractions   expansions, see \cite{Naka}. A similar analysis as the one that we performed  could be developed for  $B_{\al}$.  Note that uniform regularity results  for differences of such functions have immediate consequences on their pointwise regularity;  for example,
 $B_{1/2}-B\in C^{1/2}$, cf. Theorem 4.6 of \cite{MaMoYo}; since the $p$-exponents of $B$  belong to $[0, 1/2]$, it follows that $B_{1/2}$  shares the same $p$-exponent as $B$ (except perhaps for badly approximable points where Proposition \ref{propbad} leaves room for a  cancellation between moduli of continuity).
 
Other extensions are proposed in  \cite{LuMaNaNa} where the logarithm in $B$ is replaced by another function. An important subcase consists of choosing $1/x^\beta$ with  $0< \beta <1$. In this case the corresponding Brjuno function does not belong to all $L^p$ spaces and its pointwise exponent can be studied  for a restricted range of $p$s only.  Such a function can be seen as  a fractional derivative of the corresponding Brjuno function (defined with a logarithm); we can therefore expect that (when defined)  its $p$-exponent is $\frac{1}{\tau (x_0)} -\beta$; indeed, this would be true under the assumption that these Brjuno functions only display {\em cusp singularities}  (i.e. if the pointwise regularity exponents  of these functions are shifted by $\beta$ only after a fractional integration of order $\beta$, see \cite{Porqu}), a plausible assumption in view of Proposition  \ref{propregprimit}  which asserts that  it is the case for $B$ itself. 
 
 \smallskip  
 
 The Brjuno function   can be interpreted as the imaginary part of a complex analytic function $\mathcal{B}$, see Section 1.3 of  \cite{MaMoYo2}; a remarkable property of the real part  of $\mathcal{B}$  is that it is  a bounded function which is continuous except at rationals, where it has a left and a right limit. This property   is shared  with some  {\em Davenport series}, which are  of the form $ \sum a_n\omega (nx) $, where  $ \omega (x) =  \{ x\}  - 1/2$ if $x\in\R\setminus\Z$ and  $\omega(x)=0$ else.   If $(a_n) \in l^1$, these series   display  jumps  located at  rational numbers, thus often leading  to a pointwise regularity exponent  related with Diophantine approximation, see \cite{Dave} in which  a multifractal analysis based on the H\"older exponent is  developed, and where  discontinuities at rationals play a key role. This indicates that a multifractal   analysis may also be performed on $Re (\mathcal{B}) $:  Since, for $p \in (1,  \infty )$,  the Hilbert transform does  not modify the value of the $p$-exponents, it follows that \eqref{pexpbrj} (and hence Theorem \ref{thm:SpecB})  also holds for $Re (\mathcal{B}) $; thus all $p$-exponents of $Re (\mathcal{B}) $ coincide for $p >1$, except perhaps for $p = +\infty$. A natural conjecture therefore is that  it is also the case for $p = +\infty$, i.e. that the H\"older exponent of $Re (\mathcal{B}) $ is 
\[
\left\{ \begin{array}{ll} h_{Re (\mathcal{B}) }(x_0)=0 & \textrm{for } x_0\in\Q , \\ & \\  h_{Re (\mathcal{B})} (x_0)=\displaystyle\frac{1}{\tau(x_0)} & \textrm{ otherwise. } 
\end{array}  \right.
\]
The result clearly holds for $x_0\in\Q$, because $Re (\mathcal{B}) $ is discontinuous at rational points. Additionally, since any function satisfies $h_{f} (x_0) \leq h^p_{f} (x_0)$, it follows that  if $ x_0 \notin \Q$,  then $ h_{Re (\mathcal{B})} (x_0)\leq {1}/{\tau(x_0)}$.

\end{document}